\documentclass[11pt, reqno]{amsart}
\usepackage{syntonly}
\usepackage{pstricks,subfigure,epsfig}
\usepackage{amsmath,amsfonts,amssymb,amsthm,euscript,upgreek}
\usepackage{enumerate}
\usepackage{multirow}
\usepackage{appendix}

\newcommand{\R}{\mathbb{R}}
\newcommand{\p}{\mathbb{P}}

\newcommand{\CP}{\mathbb{C}\mathrm{P}}

\newcommand{\ol}{\mathrm{Hol}}
\newcommand{\hilb}{\mathcal{H}}

\newcommand{\f}{\rightarrow}
\newcommand{\C}{\mathbb{C}}

\newcommand{\de}{\partial}

\newcommand{\Ric}{\mathrm{Ric}}

\newcommand{\K}{K\"{a}hler}

\newtheorem{theor}{Theorem}

\newtheorem{lem}[theor]{Lemma}
\newtheorem{cor}[theor]{Corollary}

\newtheorem{ex}{Example}
\newtheorem{remar}[theor]{Remark}

\newtheorem{Aprop}{Proposition}

\newtheorem{Alem}[Aprop]{Lemma}

\def\div {\mathop{\hbox{div}}}

\begin{document}
\title[Some remarks on LeBrun's Ricci flat metrics] {Some remarks on the  K\"{a}hler  geometry of LeBrun's Ricci  flat metrics  on $\C^2$}
\author[A. Loi, M. Zedda]{Andrea Loi, Michela Zedda}
\address{Dipartimento di Matematica e Informatica, Universit\`{a} di Cagliari,
Via Ospedale 72, 09124 Cagliari, Italy}
\email{loi@unica.it; michela.zedda@gmail.com }
\author[F. Zuddas]{Fabio Zuddas}
\address{Dipartimento di Matematica, Parco Area delle Scienze 53/A  Parma (Italy)}
\email{fabio.zuddas@unipr.it}

\thanks{
The first author was supported  by the M.I.U.R. Project \lq\lq Geometric
Properties of Real and Complex Manifolds'';
the second author was  supported by  RAS
through a grant financed with the ``Sardinia PO FSE 2007-2013'' funds and
provided according to the L.R. $7/2007$.}
\date{}

\begin{abstract}
In this paper we investigate the balanced condition (in the sense of Donaldson) and the existence of an Engli\v{s} expansion for the LeBrun's metrics on $\C^2$. Our first result shows  that a LeBrun's metric on $\C^2$ is never balanced unless it is the flat metric. The second one shows that an Engli\v{s} expansion of the Rawnsley's function associated to a LeBrun's metric always exists, while the coefficient $a_3$ of the expansion vanishes if and only if the LeBrun's metric is indeed the flat one.
\end{abstract}

\maketitle

\section{Introduction}
 In \cite{LeBrun1} Claude LeBrun constructs, for all positive real numbers $m\geq 0$,
a  family of  \K\ metrics $g_m$  on ${\C}^2$,  whose associated \K\ form is given by
$\omega_m = \frac{i}{2} \partial \bar \partial \Phi_m$, where
\begin{equation}\label{Phim}
\Phi_m(u,v) = u^2 + v^2 + m (u^4 + v^4),
\end{equation}
 $u$ and $v$ are implicitly defined by
 \begin{equation}\label{x1z1}
 x_1= |z_1| =
e^{m(u^2-v^2)} u,\quad  x_2=|z_2| = e^{m(v^2-u^2)} v,
\end{equation}
and $(z_1, z_2)$ are the standard complex coordinates on $\C^2$.
 For $m = 0$ one
gets the flat metric $g_0$, i.e.  $\omega_0=\frac{i}{2}\left(dz_1\wedge d\bar z_1+dz_2\wedge d\bar z_2\right)$, while for $m>0$ each of the $g_m$'s
represents
 the first example of complete Ricci flat  and non-flat metric on ${\C}^2$ having the same
volume form of the flat metric $g_0$, namely
$\omega_m\wedge\omega_m=\omega_0\wedge\omega_0$.
 Moreover, for $m>0$,
$g_m$ is  isometric (up to dilation and rescaling)   to the
Taub-NUT  metric. In this paper  we refer to $g_m$, for a fixed $m\geq 0$, as the   {\em LeBrun  metric}.
The \K\ form $\omega_m$   has  been studied from the symplectic point of view by the first and third author of the present paper in
\cite{sympTN} (see also \cite{globsymp}). There we prove that for each $m$, $(\C^2, \omega_m)$
admits  global  Darboux coordinates. More precisely,  we construct an explicit  family $\Phi_m$ of symplectomorphisms  from $(\C^2, \omega_m)$ into $(\R^4, \omega_0)$ such that $\Phi_m^*\omega_0=\omega_m$
and $\Phi_0$ equals the identity of $\C^2\cong{\R}^4$.

In this paper we  study the  balanced condition, in the sense of Donaldson, and  Engli\v{s} expansion  for the LeBrun metric.
The main results are Theorem \ref{taubnut} and Theorem \ref{expansion} below. The first one states that the LeBrun metric on $\C^2$ is never balanced unless it is the flat metric, while the second one proves that the cefficient $a_3$ in Engli\v{s} expansion of Rawnsley's function $\epsilon_{\alpha g}$ vanishes if and only if $m=0$.

The paper is organized as follows. In the next section, after recalling what we need about balanced metrics, we state and prove Theorem \ref{taubnut}. Section \ref{englis} is devoted to Engli\v{s} expansion and to the proof of Theorem \ref{expansion}. Finally, in the Appendix we verify that the LeBrun's metrics are complete and Ricci flat (well-known facts added for completeness' sake) and we compute the norm of the curvature tensor and its Laplacian, which are needed in Section \ref{englis}.

\vskip 0.3cm

\noindent
The authors would like to thank Miroslav Engli\v{s}
for his very useful comments and remarks.

\section{On the balanced condition for the LeBrun metric}
Let $M$ be an $n$-dimensional complex manifold  endowed with a K\"ahler metric $g$
and let $\omega$ be the \K\ form associated to $g$, i.e. $\omega (\cdot ,\cdot  )=g(J\cdot, \cdot)$.
Assume that  the metric $g$ can be described by a
strictly plurisubharmonic real-valued function $\Phi :M\rightarrow\R$, called a {\em K\"ahler potential} for $g$,
i.e.
$\omega=\frac{i}{2}\de\bar\de \Phi$.

For a real number $\alpha >0$ consider   the weighted Hilbert space $\hilb_{\alpha \Phi}$ consisting of square integrable holomorphic functions on $(M, g)$, with weight $e^{-\alpha\Phi}$, namely
\begin{equation}\label{hilbertspace}
\hilb_{\alpha\Phi}=\left\{ f\in\ol(M) \ | \ \, \int_M e^{-\alpha\Phi}|f|^2\frac{\omega^n}{n!}<\infty\right\},
\end{equation}
where $\frac{\omega^n}{n!}=\det(\frac{\de^2\Phi}{\de z_\alpha \de \bar z_\beta})\frac{\omega_0^n}{n!}$ is the volume form associated to $\omega$ and $\omega_0=\frac{i}{2}\sum_{j=1}^{n} dz_j\wedge d\bar z_j$ is the standard K\"ahler form on $\C^n$.
If $\hilb_{\alpha \Phi}\neq \{0\}$ we can pick an orthonormal basis $\{f_j^{\alpha}\}$ and define its reproducing kernel
$$K_{\alpha \Phi}(x, y)=\sum_{j=0}^N f_j^{\alpha}(x)\overline{f_j^{\alpha}(y)},\qquad x,\, y\in M,$$
where $N+1$ ($N\leq\infty$) denotes the complex dimension of $\hilb_{\alpha \Phi}$.
Consider the function
\begin{equation}\label{epsilon}
\epsilon_{\alpha g}(x)=e^{-\alpha\Phi(x)}K_{\alpha \Phi}(x, x).
\end{equation}
As suggested by the notation it is not difficult to verify  that  this function  depends only on the metric $\alpha g$ and not on the choice of the K\"ahler potential $\Phi$
(which is defined up to the sum with the real part of a holomorphic function on $M$)
or on  the orthonormal basis chosen.
\vskip 0,3cm

\noindent {\bf Definition.} The metric $\alpha g$ is \emph{balanced} if  the function $\epsilon_{\alpha g}$ is a positive  constant.
\vskip 0,3cm

A balanced metric $g$ on $M$ can be viewed as  a particular  projectively induced  \K\ metric.
Recall that a  K\"{a}hler metric $g$ on a complex manifold $M$ is  {\em
projectively induced} if there exists a \K\   (i.e. a
holomorphic and isometric) immersion $F\!: M\rightarrow
{\C}P^N$, $N\leq \infty$, such that $F^*(g_{FS})=g$, where
$g_{FS}$ denotes the Fubini--Study metric on $\CP^N$.
Projectively induced K\"{a}hler metrics enjoy
important geometrical  properties and were extensively  studied in
\cite{ca} to whom we refer the reader for details).
In the case of a balanced metric, the
K\"ahler immersion
$$F_{\alpha}\!: M\f \CP^N,\ \ x\mapsto [f_0^{\alpha}(x), \dots ,f_N^{\alpha}(x)],\quad  \quad N\leq\infty,$$ is given by the  orthonormal basis $\{f_j^{\alpha}\}$ of the Hilbert space $\hilb_{\alpha \Phi}$.
Indeed the map $F_{\alpha}$ is well-defined since $\epsilon_{\alpha g}$ is a positive constant
and hence for all $x\in M$ there exists $\varphi\in \hilb_{\alpha \Phi}$ such that $\varphi(x)\neq 0$. Moreover, if $\omega_{FS}$ denotes the Fubini--Study \K\
form on $\CP^N$, then
\begin{equation}\label{balprojind}
\begin{split}
F_{\alpha}^*\omega_{FS}=&\frac{i}{2}\de\bar\de\log\sum_{j=0}^N|f_j(z)|^2
=\frac{i}{2}\de\bar\de\log K_{\alpha \Phi} (z, z)=\\
=&\frac{i}{2}\de\bar\de\log \epsilon_{\alpha g}+\frac{i}{2}\de\bar\de\log e^{\alpha\Phi}
=\frac{i}{2}\de\bar\de\log \epsilon_{\alpha g}+   \,\alpha\omega.
\end{split}
\end{equation}
Hence if $\alpha g$ is balanced then it is projectively induced via the holomorphic  map  $F_{\alpha}$.

In the literature the function $\epsilon_{\alpha g}$ was first introduced under the name of $\eta$-{\em function} by J. Rawnsley in \cite{rawnsley}, later renamed as $\theta$-{\em function} in \cite{CGR}.
 If $\epsilon_{\alpha g}$  is balanced for all sufficiently large
$\alpha$  then the  geometric quantization (namely the holomorphic line bundle such that $c_1(L)=[\omega]$) is called regular. Regular quantizations
plays a prominent role in the theory of quantization by deformation
of K\"{a}hler manifolds developed in \cite{CGR} (there the map  $F_{\alpha}$ is called  the  {\em coherent states map}).

\begin{remar}\rm
The definition of balanced metrics
 was originally given by S. Donaldson \cite{donaldson} (see also \cite{arezzoloi}, \cite{balancedC} and \cite{regcov}) in the case of  a compact polarized \K\ manifold $(M, g)$ and generalized in \cite{arezzoloi} (see also \cite{globsymp}, \cite{englisweigh}, \cite{grecoloi},  \cite{hartogsbalanced} \cite{cartanbalanced}) to the noncompact case. Here we give only the definition for those \K\ metrics which admit a globally defined potential such as the family of metrics $g_m$ on $\C^2$ in this paper.
 \end{remar}

We describe now a well-known example which is the prototype of our analysis.

\begin{ex}\label{c0}\rm
Let $(\C^n, g_0)$ be the complex Euclidean space endowed with the flat metric $g_0$, whose associated K\"ahler form is given by
$$\omega_0=\frac{i}{2}\de\bar\de||z||^2,$$
where $\Phi_0= ||z||^2=\sum_{j=1}^n|z_j|^2$. Let $\hilb_{\Phi_0}$ be the weighted Hilbert space of squared integrable holomorphic functions on $(\C^n,  g_{0})$, with weight  $e^{-||z||^2}$, namely
$$\hilb_{\Phi_0}=\left\{ \varphi\in\ol(\C^n)\ | \  \int_{\C^n}e^{-||z||^2}|\varphi|^2 \frac{\omega_{0}^n}{n!}<\infty \right\}.$$
Since the potential depends on $r_j = |z_j|^2$, $j = 1, \dots, n$, an orthogonal basis for $\hilb_{\Phi_0}$ is given by the monomials $z_1^{j_1}\cdots z_n^{j_n}$.
Furthermore, by
\begin{equation}
\begin{split}
&\int_\C |z_1|^{2j_1}\cdots |z_n|^{2j_n}e^{-  ||z||^2}\frac{i^n}{2^n}dz_1\wedge d\bar z_1\cdots dz_n\wedge d\bar z_n \\
&=\pi^n\int_0^\infty \cdots \int_0^\infty r_1^{j_1}\cdots r_n^{j_n} e^{-(r_1+\dots +r_n)}dr_1\cdots dr_n\\
&={\pi^n\, j_1!\cdots j_n!},
\end{split}
\end{equation}
an orthonormal  basis for  $\hilb_{\Phi_0}$ is given by $\left\{\frac{ z_1^{j_1}\cdots z_n^{j_n} }{\sqrt{\pi^n}\,j_1!\cdots j_n!}\right\}$. The reproducing kernel of
$\hilb_{\Phi_0}$  reads
$$K_{\Phi_0}(z, z)=\frac{ 1 }{\pi^n}\,e^{  ||z||^2},$$
hence $\epsilon_{ g_0}=\frac{ 1 }{\pi^n}$, and so
$g_0$ is  a balanced metric. Moreover, since for all $\alpha >0$,  the metric   $\alpha g_0$ is holomorphically isometric to $g_0$ it follows that
also $\alpha g_0$ is  a balanced metric.
\end{ex}

\begin{remar}\rm\label{remarmultiple}
If  $g$ is a balanced (resp. projectively induced) metric  on a complex manifold  $M$ and $\alpha$ is a positive constant  then there is not any reason for the metric
$\alpha g$
to be balanced (resp. projectively induced), as it happens in the previous example.  Consider for istance a Cartan domain endowed with its Bergman metric $g_B$. In accordance with the choice of $\alpha>0$, $\alpha g_B$ is either balanced or projectively induced but not balanced or also not projectively induced.  
In \cite{articwall} and
 \cite{cartanbalanced}
 the first and second author of this paper
study the balanced and projectively induced condition for multiples of the Bergman metric $g_B$ on each  Cartan domain in terms of its genus
$\gamma$.
\end{remar}

Here we study the balanced condition for the LeBrun metric. Consider the weighted Hilbert space
$\hilb_{\alpha\Phi_m}$ for the  metric $\alpha g_m$, $\alpha>0$
(where $\Phi_m$ is given by (\ref{Phim})), namely
$$\hilb_{\alpha\Phi_m}=\left\{f\in\ol(\C^2)\ | \ \int_{\C^2} e^{-\alpha\Phi_m} |f|^2 \frac{\omega_0^2}{2}\right\},$$
where we have taken into account that  $\omega_m\wedge\omega_m=\omega_0\wedge\omega_0=\omega_0^2$.
The following lemma is needed in the proof of our main results.
\begin{lem}\label{z1z2k}
The monomials  $\{z_1^j z_2^k\}$, $j, k=1, \dots $  are a complete orthogonal system for the Hilbert space
$\hilb_{\alpha\Phi_m}$.  Moreover Rawnsley's function is given by:

\begin{equation}\label{repkern}
\begin{split}
\epsilon_{\alpha g}&=e^{-\alpha\Phi_m}K_{ \alpha\Phi_m} =e^{-\alpha\Phi_m}\sum_{j,k=0}^{+ \infty} \frac{|z_1|^{2j}|z_2|^{2k}}
{\| z_1^j z_2^k \|_{m, \alpha}^2}\\
&=e^{-\alpha\Phi_m}
\sum_{j,k=0}^{+ \infty} \frac{e^{2m(j-k)(U-V)} U^j V^k}
{\| z_1^j z_2^k \|_{m, \alpha}^2}.
\end{split}
\end{equation}
where $U = u^2$, $V = v^2$,
\begin{equation}\label{norma}
\begin{split}
\| z_1^j z_2^k \|^2_{m, \alpha}=&\int_{\C^2}e^{-\alpha \left(u^2+v^2+m(u^4+v^4)\right)}|z_1|^{2j} |z_2|^{2k}\frac{\omega_0^2}{2}\\
=&\, \pi^2\left[\, I_{m, \alpha}(j,j,k) I_{m, \alpha}(k,k,j) + 2m I_{m, \alpha}(j+1,j,k ) I(k,k,j) +\right. \\
&\left.+2m  I_{m, \alpha}(j,j,k ) I_{m, \alpha}(k+1,k,j)\right],
\end{split}
\end{equation}
and
\begin{equation}\label{ijk}
I_{m, \alpha}(i,j,k) = \int_{0}^{\infty} e^{- \alpha [U + m U^2]+ 2mU(j -k)}U^i dU.
\end{equation}
\end{lem}
\begin{proof}
Since the metric depends only on the squared module of the variables it is easy to see that the monomials $\{z_1^j z_2^k\}$, $j, k=1, \dots $ are a complete orthogonal system for $\hilb_{\alpha\Phi_m}$  (cfr. \cite[Lemma 3.11]{englquant}).
By passing to polar coordinates we have
$$\| z_1^j z_2^k \|^2_{m, \alpha} = \pi^2 \int_{0}^{\infty} \int_{0}^{\infty} e^{- \alpha [u^2 + v^2 + m( u^4 +
v^4)]} r_1^{j} r_2^{k} dr_1 dr_2$$
where $u = u(r_1, r_2), v = v(r_1, r_2)$. Set  $u^2 = U$, $v^2 = V$ and consider the map
$$G:
{\R}^2 \rightarrow {\R}^2, \ (U, V) \mapsto  (r_1=e^{2m(U-V)} U,\
r_2=e^{2m(V-U)} V)$$ and its Jacobian matrix
\begin{displaymath}
J_G = \left( \begin{array}{cc}
(1 + 2mU) \; e^{2m(U-V)} &  - 2 m U \; e^{2m(U-V)} \\
- 2 m V \; e^{2m(V-U)} &  (1 + 2m V) \; e^{2m(V-U)}
\end{array} \right) .
\end{displaymath}
Since $\det(J_G) = 1 + 2 m (U + V) \neq 0$, by a change of coordinates, the above integral becomes
\begin{equation}
\begin{split}
&\| z_1^j z_2^k \|^2_{m, \alpha} =\\
=&\,\pi^2\!\! \int_{0}^{\infty}\!\!\!\! \int_{0}^{\infty}\!\!\!\!\! e^{- \alpha [U + V + m( U^2 +V^2)]} e^{2m j(U-V)} e^{2m k (V-U)} U^j V^k [1 + 2 m (U\! +\! V)] dUdV\\
=&\,\pi^2\!\! \int_{0}^{\infty} \int_{0}^{\infty} e^{- \alpha [U + V + m( U^2 +
V^2)]} e^{2m j(U-V)} e^{2m k (V-U)} U^j V^k dU dV +\\
&+ 2 m \pi^2 \int_{0}^{\infty} \int_{0}^{\infty} e^{- \alpha [U + V
+ m( U^2 + V^2)]} e^{2m j(U-V)} e^{2m k (V-U)} U^{j+1} V^k dU dV
+\\
&+2 m \pi^2 \int_{0}^{\infty} \int_{0}^{\infty} e^{- \alpha [U + V
+ m( U^2 + V^2)]} e^{2m j(U-V)} e^{2m k (V-U)} U^{j} V^{k+1}dU dV\nonumber
\end{split}
\end{equation}
\begin{equation}
\begin{split}
=&\,\pi^2\!\! \int_{0}^{\infty} e^{- \alpha [U + m U^2]+ 2mU(j -k)} U^j dU \int_{0}^{\infty} e^{-
\alpha [V + m V^2] + 2mV (k-j)} V^k dV +\\
&+ 2m \pi^2\!\! \int_{0}^{\infty} \!\!\!\!e^{- \alpha [U + m U^2]+ 2mU(j -k)} U^{j+1} dU \int_{0}^{\infty} e^{-
\alpha [V + m V^2] + 2mV (k-j)} V^k dV +\\
&+ 2m \pi^2\!\! \int_{0}^{\infty} e^{- \alpha [U + m U^2]+ 2mU(j -k)} U^j dU \int_{0}^{\infty} e^{-
\alpha [V + m V^2] + 2mV (k-j)} V^{k+1} dV.\nonumber
\end{split}
\end{equation}
Formula (\ref{norma}) and hence (\ref{repkern}) follows by  comparing last equality with (\ref{ijk}).
\end{proof}

We are now in the position to state    the first  main result of this paper.

\begin{theor}\label{taubnut}
Let $m\geq 0$,   $g_m$ be the  LeBrun metric on $\C^2$ and $\alpha$ be a positive real number. Then the  metric $\alpha g_m$  is balanced if and only if $m=0$, i.e. it is the flat metric on $\C^2$.
\end{theor}

In order to prove it we need  the following lemma, interesting on its own sake.

\begin{lem}\label{pind}
Let $m\geq 0$,   $g_m$ be the  LeBrun metric on $\C^2$ and $\alpha$ be a positive real number. Then  $\alpha g_m$   is not projectively induced for   $m>\frac{\alpha}{2}$.
\end{lem}
\begin{proof}
Assume by contradiction that $\alpha g_m$ is projectively induced, namely that  there exists $N\leq\infty$ and   a \K\  immersion
of $(\C^2, \alpha g_m)$ into $\C P^N$. Then, it does exist also a \K\ immersion into $\CP^N$
of  the   \K\ submanifold of $(\C^2, \omega_m)$  defined by $z_2 = 0$, $z_1 = z$, endowed with the induced metric, having potential $\tilde \Phi_m = u^2 + m u^4$, where $u$ is defined implicitly by $z\bar{z} = e^{2 m u^2} u^2$. Observe that $\tilde \Phi_m$ is the Calabi's diastasis function for this  metric , since it is a rotation invariant potential centered at the origin (see \cite{ca} or also \cite[Th. 3, p. 3]{articwall}).

Consider the power expansion around the origin of the function $e^{\alpha\tilde \Phi_m}-1$, that, by (\ref{Phim}) and (\ref{x1z1}), reads
\begin{equation}
e^{\alpha\tilde\Phi_m}-1=\alpha |z|^2+\frac{\alpha}{2}(\alpha-2 m)|z|^4+\dots.\nonumber
\end{equation}
Since $\alpha-2 m\geq0$ if and only if $m\leq\frac{\alpha}{2}$, it follows by Calabi's criterion (see \cite[Th. 6, p. 4]{articwall}) that $\alpha g_m$ can not admit a K\"ahler immersion into  $\CP^N$ for any $m>\frac{\alpha}{2}$.
\end{proof}

\begin{remar}\rm
We conjecture that $\alpha g_m$ is not projectively induced also for $0<m\leq \frac{\alpha}{2}$ (for $m =0$ it is projectively induced and also balanced by Example \ref{c0}).
We have computer evidences of this fact but we still do not have a proof. Nevertheless we do not need  this fact in the following construction. It is also worth pointing out that using Calabi's techniques \cite{ca} it is a simple matter to verify that $(\C^2, \alpha g_m)$ cannot be \K\ immersed into the complex hyperbolic space $\C H^N$, $N\leq \infty$, equipped with the hyperbolic metric and, moreover, if $(\C^2,  g_m)$ admits a \K\ immersion into
$(\C^N, g_0)$, $N\leq\infty$, then $m=0$.
\end{remar}

\noindent
\begin{proof}
[Proof of Theorem \ref{taubnut}.]
If $m=0$, $\alpha\,g_0$ is the flat metric on $\C^2$ which is balanced for all positive values of $\alpha$ (cfr. Example \ref{c0}). Furthermore, for $m>\frac{\alpha}{2}$, by Lemma \ref{pind}, $\alpha g_m$ is not projectively induced and hence it cannot be balanced. Thus the proof of the theorem will be achieved if we show that $\alpha g_m$ is not balanced for  $0<m\leq \frac{\alpha}{2}$. Notice that if  $\alpha g_m$ is balanced
then there exists a basis $\{f_j^{\alpha} \}$ for ${\hilb}_{\alpha \Phi_m}$ such that
\begin{equation}\label{balancedTN}
\int_{\C^2}e^{-\alpha \left(u^2+v^2+m(u^4+v^4)\right)}f_j^{\alpha}\bar f_k^{\alpha}\frac{\omega_0^2}{2}=\lambda\delta_{jk},
\end{equation}
for some $\lambda$ not depending on $j$, $k$.
Notice also that
the power expansion around the origin of the function $e^{\alpha\Phi_m}$ reads
\begin{equation}
e^{\alpha\Phi_m}=1+\alpha(|z_1|^2+|z_2|^2)+ \alpha\,\left(\frac{\alpha}{2}-\,m\right)\left(|z_1|^4+|z_2|^4\right)+ \alpha ^2\,|z_1|^2\,|z_2|^2+\dots.\nonumber
\end{equation}
as it follows  by (\ref{Phim}) and (\ref{x1z1}). Therefore,
since the metric $g_m$ depends on the module of the variables, we can assume without loss of generality that $f_0^{\alpha}=1$ and $f_1^{\alpha}=\alpha\,z_1z_2$.
We are going to show that  $\|f_0^{\alpha}\|^2_{m, \alpha}= \|f_1^{\alpha}\|^2_{m, \alpha}$ does not hold for $0<m\leq\frac{\alpha}{2}$  and hence,
for these values of $m$, $\alpha g_m$ is not balanced.
By (\ref{norma}) we have
\begin{equation}
\|f_0^{\alpha}\|^2_{m, \alpha}=\int_{\C^2}e^{-\alpha \Phi_m}\frac{\omega_0^2}{2}=\pi^2\left[\, I_{m, \alpha}(0,0,0)^2 + 4m I_{m, \alpha}(1,0,0 ) I_{m, \alpha}(0,0,0) \right],
\end{equation}
\begin{equation}
\begin{split}
\|f_1^{\alpha}\|^2_{m, \alpha}&=
\alpha^2\int_{\C^2}e^{-\alpha \Phi_m}|z_1|^2\,|z_2|^2\frac{\omega_0^2}{2}\\
&=\alpha^2\pi^2\left[\, I_{m, \alpha}(1,1,1)^2 + 4m I_{m, \alpha}(2,1,1 ) I_{m, \alpha}(1,1,1) \right].
\end{split}
\end{equation}
For all $\alpha>0$ define $h_{\alpha}\!:\R^+\cup\{0\}\f \R$ as
\begin{equation}
\begin{split}
h_{\alpha}(m)=&\,I_{m, \alpha}(0,0,0)^2 + 4m I_{m, \alpha}(1,0,0 ) I_{m, \alpha}(0,0,0)+\\
&-\alpha^2I_{m, \alpha}(1,1,1)^2 - 4m \alpha^2 I_{m, \alpha}(2,1,1 ) I_{m, \alpha}(1,1,1).\nonumber
\end{split}
\end{equation}

We need to prove that $h_{\alpha}(m)\neq 0$ for  $0<m\leq\frac{\alpha}{2}$.
Observe  that $h_{\alpha}(0)=0$, in accordance with the fact that the flat metric $g_0$ on $\C^2$ is balanced for all $\alpha >0$. We have:
\begin{equation}
I_{m, \alpha}(0,0,0)= \int_{0}^{\infty} e^{- \alpha (U + m U^2)}dU,
\end{equation}
\begin{equation}
\begin{split}
I_{m, \alpha}(1,0,0)= I_{m, \alpha}(1,1,1)=&\int_{0}^{\infty} e^{- \alpha (U + m U^2)}UdU\\
=&-\frac{1}{2 m}I_{m, \alpha}(0,0,0)+{\frac {1}{2\alpha m}},
\end{split}
\end{equation}
\begin{equation}
\begin{split}
I_{m, \alpha}(2,1,1)=&\int_{0}^{\infty} e^{- \alpha (U + m U^2)}U^2dU\\
=&\frac{2m+\alpha}{4m^2 \alpha}I_{m, \alpha}(0,0,0)-\frac {1}{4m^2\alpha}.
\end{split}
\end{equation}
Hence
\begin{equation}
\begin{split}
h_{\alpha}(m)=&\frac{1}{4m^2\alpha}(-4 m^2\alpha \,I_{m, \alpha}(0,0,0)^2 + 8\,m^2I_{m, \alpha}(0,0,0) + \alpha^3 I_{m, \alpha}(0,0,0)^2+  \\
&+4 \alpha^2 m \,I_{m, \alpha}(0,0,0)^2  + \,\alpha - 4 m\alpha \, I_{m, \alpha}(0,0,0) - 2\,\alpha^2 I_{m, \alpha}(0,0,0)).\nonumber
\end{split}
\end{equation}
It follows that $h_{\alpha}(m)=0$ if and only if $I_{m, \alpha}(0,0,0)$ satisfies the following second order equation
\begin{equation}
y^2(-4m^2 \alpha + \alpha^3 +4m \alpha^2)+2y(4 m^2 - 2 m \alpha - \alpha^2)+ \alpha =0.\nonumber
\end{equation}
Its discriminant
\begin{equation}
\Delta= (4 m^2 - 2 m \alpha -  \alpha^2)^2 - \alpha^2 (-4m^2 + \alpha^2 +4m \alpha) = 16m^3(m - \alpha)\nonumber
\end{equation}
is negative for all  $0<m<\alpha$, so that $h_{\alpha}(m)\neq 0$ for all $0 < m < \alpha$ and, a fortiori,  for $0<m\leq\frac{\alpha}{2}$,
which is exactly what we wanted to prove.
\end{proof}

\section{Engli\v{s} expansion for the LeBrun metric}\label{englis}
Let $(M, g, \omega =\frac{i}{2}\partial\bar\partial\Phi)$ be a \K\ manifold as in the previous
section (i.e. admitting a globally defined  \K\ potential) and consider Rawnsley's function (\ref{epsilon}) for a fixed $\alpha>0$.
Even if this function is not constant (i.e. $\alpha g$ is not balanced) it is interesting to understand
when   $\epsilon_{\alpha g}$ admits an asymptotic expansion for $\alpha\rightarrow+\infty$.
For example this  turns out to be true
when $M$ is  a strongly pseudoconvex  bounded domain in
${\C}^n$ with real analytic boundary or when $M$ is an Hermitian symmetric space of noncompact type equipped with its Bergman metric (cfr. Remark \ref{remarmultiple} above).

Indeed M. Engli\v{s} \cite{englis2}
shows that for these domains
$\epsilon_{\alpha g}$ admits the following  asymptotic expansion (which has been  christened  as {\em Engli\v{s} expansion} in \cite{engexp})
with respect to $\alpha$
\begin{equation}\label{Talpha}
\epsilon_{\alpha}(x) \sim
\sum_{j=0}^{+\infty} a_j(x){\alpha}^{n-j}
\end{equation}
where $a_0=1$ and $a_j$ for $j=0,1, 2,\ldots$, are smooth coefficients, in the sense that, for every integers $l, r$ and every compact $H \subseteq M$,
\begin{equation}\label{TalphaMEANS}
\| \epsilon_{\alpha}(x) -
\sum_{j=0}^l  a_j(x){\alpha}^{n-j} \|_{C^r} \leq \frac{C(l, r, H)}{\alpha^{l+1}}
\end{equation}
for some constant $C(l, r, H) >0$.
 Moreover, in
\cite{me2}  M. Engli\v{s} also   computes  the coefficients $a_j,\ j\leq
3$, namely:
\begin{equation}\label{coefflu}
\left\{\begin{array}
{l}
a_0=1\\
a_1(x)=\frac{1}{2}\rho\\
a_2(x)=\frac{1}{3}\Delta\rho
+\frac{1}{24}(|R|^2-4|\Ric |^2|+3\rho ^2)\\
a_3(x)=\frac{1}{8}\Delta\Delta\rho +
\frac{1}{24}\div\div (R, \Ric)-
\frac{1}{6}\div\div (\rho\Ric)+\\
\qquad\quad\ \ +\frac{1}{48}\Delta (|R|^2-4|\Ric |^2|+8\rho ^2)+
\frac{1}{48}\rho(\rho ^2- 4|\Ric |^2+ |R|^2)+\\
\qquad\quad\ \ +\frac{1}{24}(\sigma_3 (\Ric)- \Ric (R, R)-R(\Ric ,\Ric)),
\end{array}\right.
\end{equation}
where
\noindent (see also \cite{loismooth} and \cite{lutian})
$\rho$, $R$, $Ric$ denote respectively the scalar curvature,
the curvature tensor and the Ricci tensor of $(M, g)$,
and we are using  the following notations (in local coordinates $z_1, \dots , z_n$):
\begin{equation}\begin{array}{l}\label{values}
|D^{'}\rho|^2=\sum_{i=1}^n|\frac{\partial \rho}{\partial z_i}|^2,\\
|D^{'}\Ric|^2=\sum_{i, j, k=1}^n|\Ric_{i\bar j,k}|^2,\\
|D^{'}R|^2=\sum_{i, j, k, l, p=1}^n|R_{i\bar jk\bar l,p}|^2,\\
\div\div (\rho \Ric)=2|D^{'}\rho|^2+
\sum_{i,j=1}^n\Ric_{i\bar j}\frac{\partial^2 \rho}{\partial\bar z_j\partial z_i}
+\rho\Delta\rho,\\
\div\div (R, \Ric)=
-\sum_{i,j=1}^n\Ric_{i\bar j}\frac{\partial^2 \rho}{\partial\bar z_j\partial z_i}
-2|D^{'}\Ric|^2+\\
\qquad\qquad\qquad\quad\ \ +\sum_{i, j, k, l=1}^nR_{j\bar il\bar k}R_{i\bar j,k\bar l}-
R(\Ric, \Ric)-\sigma_3(\Ric),\\
R(\Ric, \Ric)=\sum_{i, j, k, l=1}^nR_{i\bar jk\bar l}\Ric_{j\bar i}\Ric_{l\bar k},\\
\Ric(R, R)=\sum_{i, j, k, l,p,q=1}^n\Ric_{i\bar j}R_{j\bar kp\bar q}R_{k\bar iq\bar p},\\
\sigma_3 (\Ric)=\sum_{i, j, k=1}^n\Ric_{i\bar j}\Ric_{j\bar k}\Ric_{k\bar i}.\\
\end{array}\end{equation}
Notice that the term $a_1$ was already computed by Berezin in his
seminal paper \cite{Ber1} on quantization by deformation. Actually
Engli\v{s} expansion  is strongly related to Berezin
transform and to asymptotic expansion of Laplace integrals
(see \cite{englis2} and   \cite{loismooth}).
It is important to point out that for a general noncompact manifold
there is not a general theorem which assures the existence of Engli\v{s} expansion
(see \cite{graloi} for the case of the Kepler manifold). A partial result in this direction is
Theorem 6.1.1 due to X. Ma and G. Marinescu \cite{mama}, which translated in our situation and with our notations reads as:

\begin{theor}\label{marinescuma}
Let $(M, g, \omega =\frac{i}{2}\de\bar\de\Phi)$ be a complete \K\ manifold  and, for $\alpha>0$, $K_{\alpha\Phi}$ be the reproducing kernel of the space

\begin{equation}
\hilb_{\alpha \Phi}=\left\{ f\in\ol(M) \ | \ \, \int_M e^{-\alpha \Phi}|f|^2 \frac{\omega^n}{n!} <\infty\right\}.
\end{equation}

\noindent Then $\epsilon_{\alpha\Phi}=e^{-\alpha\Phi}K_{\alpha\Phi}$ admits an asymptotic expansion in $\alpha$ with coefficients given by (\ref{coefflu}) provided there exists $c > 0$ such that
$$ i  R^{det} > - c \Theta$$
where $R^{det}$ denotes the curvature of the connection on $\det(T^{(1,0)} M)$ induced by $g$.
\end{theor}

Engli\v{s} expansion is the counterpart of the celebrated TYZ (Tian-Yau-Zelditch) expansion
of Kempf's distortion function
$$T_{\alpha g}(x) \sim
\sum_{j=0}^\infty  b_j(x)m^{n-j},\quad \alpha=0,1,2,\dots,$$
 for polarized compact \K\ manifolds $M$ (see  \cite{ze}
and  also \cite{arlquant}),
where  $b_j$, $j=0,1, \ldots$, are smooth coefficients with $b_0(x)=1$.
Z. Lu \cite{lu}, by means of  Tian's peak section method,
proved  that  each of the coefficients $b_j(x)$
 is a polynomial of the curvature  of the metric $g$ and its covariant
derivatives at $x$. Such a polynomial can be
found by finitely many steps of algebraic operations. Furthermore
$b_j(x)$, $j=1, 2, 3$ have the same values of those given by
(\ref{coefflu}), namely $b_j(x)=a_j(x)$, $j=1, 2, 3$.
(see also \cite{loianal} and \cite{loismooth} for the
computations of the coefficients $b_j$'s through Calabi's
diastasis function).

Due to Donaldson's work (cfr. \cite{donaldson},
\cite{do2}) in the compact case and respectively to the theory of
quantization  in the noncompact case (see, e.g. \cite{Ber1} and \cite{CGR}),
 it is  natural to study metrics with the coefficients $b_k$'s of  the  TYZ expansion
 (resp.    $a_k$'s of  Engli\v{s} expansion)  being prescribed. In particular, the vanishing of these coefficients for large enough $k$ turns out to be related to some important problems in the theory of pseudoconvex manifolds. For instance, Lu and Tian (\cite{lutian}) showed that the $b_j$'s vanish for $j > n$ provided that the Bergman kernel $K$ of the pseudoconvex domain $D$ given by the disc bundle of a polarization of $M$, has vanishing log term. Let us recall that this means that $b \equiv 0$ in the decomposition
\begin{equation}\label{fefferman}
K(x,x) = \frac{a(x)}{\psi(x)^{n+1}} + b(x) \log \psi(x) ,
\end{equation}
proved by Fefferman in his celebrated paper \cite{F} for the Bergman kernel $K$ of every bounded strictly pseudoconvex domain $\Omega = \{\psi > 0 \} \subseteq {\C}^n$ given by a defining function $\psi$ (for analogous results in the case of Szeg\"{o} kernels or weighted Bergman kernels, see respectively \cite{lutian} and \cite{englis3}). The vanishing of the log term is a condition of considerable interest both from an analytic and a geometric point of view. For example, Ramadanov (\cite{R}) conjectured that if the log term of the Bergman kernel vanishes, then the domain is biholomorphic equivalent to the unit ball (this was proved to be true for some special cases, such as domains in ${\C}^2$ or domains with rotational symmetries, see for example \cite{BoCo}, \cite{Bout}); Lu and Tian conjectured that if the log term of the Szeg\"{o} kernel of the unit circle bundle of a polarisation of $({\C} {\p}^n, \omega)$ vanishes, for some $\omega$ cohomologous to the Fubini-Study form $\omega_{FS}$, then $\omega$ is holomorphically isometric to $\omega_{FS}$. Observe that for $n=1$ (see \cite[Theorem 1.1]{lutian}) the proof of this conjecture follows by the above-mentioned result of Lu and Tian and by the expression of $a_2=b_2$  in  (\ref{coefflu}) which imply that  the  scalar curvature $\rho$  is  constant.

The following theorem is the second main result  of this paper.

\begin{theor}\label{expansion}
Let $m\geq 0$,   $g_m$ be the  LeBrun metric on $\C^2$ and $\alpha$ be a positive real number.
Then Engli\v{s} expansion of   Rawnsley's function $\epsilon_{\alpha g}$ exists and the coefficients $a_j$ are given by  (\ref{coefflu}).
Moreover,  if $a_j = 0$ for $j > 2$  then $m=0$, i.e. $g_m=g_0$ is the flat metric on $\C^2$.
\end{theor}
\begin{proof}
The first part of the statement follows from Theorem \ref{marinescuma} applied to $M = {\C}^2$ and $g = g_m$, since in this case  $i R^{det}$ equals the Ricci form of $g$ which vanishes by the Ricci-flatness of LeBrun's metrics.

Assume now that $a_j = 0$ vanishes for $j > 2$. In particular, $a_3 = 0$.
By (\ref{coefflu}),  (\ref{values}) and  by the fact that $g_m$  is Ricci flat it follows that   the Laplacian of the norm of its curvature tensor vanishes, i.e.  $\Delta |R|^2 =0$. On the other hand   formula (\ref{L|R|^2}) in Lemma \ref{curvature} yields $m=0$ and we are done.
\end{proof}
Combining the previous theorem with formula  (\ref{balprojind})  we get the following corollary which  should  be compared with
Theorem 4.1 in  \cite{tian}, where it is proven  the analogous result for \K --Einstein metrics with Einstein constant
$-1$.

\begin{cor}
Let $g_m$ be the  LeBrun metric on $\C^2$, $m\geq 0$.
Then $g_m$ can be approximated by  suitable normalized  projectively induced \K\ metrics on $\C^2$.
\end{cor}
\begin{proof}
Let  $F_{\alpha}:\C^2\rightarrow\CP^{\infty}$ be the coherent states map, namely the holomorphic map
constructed with the orthonormal basis of momomials given by  Lemma \ref{z1z2k}.  Then, by formula
(\ref{balprojind}), since $a_0=1$, one gets
 $\lim_{\alpha\rightarrow +\infty}\frac{1}{\alpha}F_{\alpha}^*g_{FS}=g$.
\end{proof}

\appendix
\section{Some computation}
\begin{Aprop}\label{wellknown}
Fix $m\geq 0$ and let $g_m$ be the LeBrun  metric on $\C^2$. Then $g_m$ is complete and $\omega_m\wedge\omega_m=\omega_0\wedge\omega_0$.
\end{Aprop}
\begin{proof}
Consider the K\"ahler potential
$$\Phi=\Phi_m(U,V)=U+V+m(U^2+V^2)$$ for the LeBrun's  metric $g_m$,
where $U=u^2$ and $V=v^2$. Here $U$ and $V$ are implicitely defined by
\begin{equation}\label{x_1=}
x_1=|z_1|^2=e^{2m(U-V)}U,
\end{equation}
\begin{equation}\label{x_2=}
x_2=|z_2|^2=e^{2m(V-U)}V.
\end{equation}
(cfr. (\ref{Phim}) and (\ref{x1z1})).
Since $\Phi=\Phi(x_1,x_2)$ (i.e. depends on the squared module of the variables), we can write the matrix $g_m =(g_{\alpha\bar\beta})$ as
\begin{equation}\label{metricg}
\left(
\begin{array}{cc}
g_{1\bar 1} & g_{1\bar 2}   \\
g_{2\bar 1}  & g_{2\bar 2}
\end{array}
\right)=
\left(
\begin{array}{cc}
\Phi_{11}\,x_1+\Phi_{1}  &\Phi_{12}\,\bar z_1 z_2    \\
 \Phi_{12}\, z_1\bar z_2 & \Phi_{22}\,x_2+\Phi_2
\end{array}
\right),
\end{equation}
where we write $\Phi_i=\de \Phi/\de x_i$ and $\Phi_{ij}=\de^2 \Phi/\de x_i \de x_j$. 
Now  observe that
$$\frac{\de \Phi}{\de x_i}=\frac{\de \Phi}{\de U}\frac{\de U}{\de x_i}+\frac{\de \Phi}{\de V}\frac{\de V}{\de x_i}, \qquad i=1,2,$$
and
\begin{equation}
\left(
\begin{array}{cc}
\frac{\de U}{\de x_1}  &\frac{\de U}{\de x_2}    \\
\frac{\de V}{\de x_1} &\frac{\de V}{\de x_2}
\end{array}
\right)=\left(
\begin{array}{cc}
\frac{\de x_1}{\de U}  &\frac{\de x_1}{\de V}    \\
\frac{\de x_2}{\de U} &\frac{\de x_2}{\de V}
\end{array}
\right)^{-1}.\nonumber
\end{equation}
By
$$\frac{\de x_1}{\de U}=(1+2mU)e^{2m(U-V)},\quad \frac{\de x_1}{\de V}=-2mUe^{2m(U-V)},$$
$$\frac{\de x_2}{\de V}=(1+2mV)e^{2m(V-U)},\quad \frac{\de x_2}{\de U}=-2mve^{2m(V-U)},$$
since $1+2m(U+V)\neq 0$, we get
\begin{equation}
\left(
\begin{array}{cc}
\frac{\de x_1}{\de U}  &\frac{\de x_1}{\de V}    \\
\frac{\de x_2}{\de U} &\frac{\de x_2}{\de V}
\end{array}
\right)^{-1}=
\left(
\begin{array}{cc}
\frac{(1+2mV)e^{2m(V-U)}}{1+2m(U+V)} &\frac{2me^{2m(U-V)}U}{1+2m(U+V)}   \\
\frac{2me^{2m(V-U)}V}{1+2m(U+V)} &\frac{(1+2mU)e^{2m(U-V)}}{1+2m(U+V)}
\end{array}
\right).\nonumber
\end{equation}
Thus, since $\de\Phi/\de U=1+2m\,U$ and $\de\Phi/\de V=1+2m\, V$,
a direct computation gives
\begin{equation}
\Phi_1=\frac{\de\Phi}{\de x_1}=(1+2mV)e^{2m(V-U)},\nonumber
\end{equation}
\begin{equation}
\Phi_2=\frac{\de\Phi}{\de x_2}=(1+2mU)e^{2m(U-V)},\nonumber
\end{equation}
\begin{equation}
\Phi_{11}=\frac{\de\Phi_1}{\de U}\frac{\de U}{\de x_1}+\frac{\de\Phi_1}{\de V}\frac{\de V}{\de x_1}=-\frac{2me^{4m(V-U)}}{1+2m(U+V)},\nonumber
\end{equation}
\begin{equation}
\Phi_{12}=\frac{\de\Phi_1}{\de U}\frac{\de U}{\de x_2}+\frac{\de\Phi_1}{\de V}\frac{\de V}{\de x_2}=\frac{4m(1+m(U+V))}{1+2m(U+V)},\nonumber
\end{equation}
\begin{equation}
\Phi_{22}=\frac{\de\Phi_2}{\de U}\frac{\de U}{\de x_2}+\frac{\de\Phi_2}{\de V}\frac{\de V}{\de x_2}=-\frac{2me^{4m(U-V)}}{1+2m(U+V)}.\nonumber
\end{equation}
Combining these formulas with (\ref{x_1=}), (\ref{x_2=}) and (\ref{metricg}) above, we get
\begin{equation}\label{metricfinal}
g_m=\left(
\begin{array}{cc}
\frac{1+4mV(1+mU+mV)}{1+2m(U+V)}e^{2m(V-U)}  &\frac{4m(1+mU+mV)}{1+2m(U+V)}z_2 \bar z_1 \\
\frac{4m(1+mU+mV)}{1+2m(U+V)}z_1\bar z_2 & \frac{1+4mU(1+mU+mV)}{1+2m(U+V)}e^{2m(U-V)}
\end{array}
\right),
\end{equation}
from which follows easily, substituting again (\ref{x_1=}) and (\ref{x_2=}), that $\det (g_m)=1$, i.e.
$\omega_m\wedge \omega_m=\det (g_m)\omega_0\wedge\omega_0=\omega_0\wedge\omega_0$.

In order to prove the completeness of $g_m$, consider the metric $\tilde g_m$ on $\C^2$ defined by
$$\tilde g_m=\frac{1}{1+2m (U+V)}\left(e^{2m(V-U)}dz_1d\bar z_1+e^{2m(U-V)}dz_2d\bar z_2\right).$$
We claim that $g_m\geq \tilde g_m$. In fact, by (\ref{metricfinal}) given $(w_1,w_2)\in T_{(z_1,z_2)}\C^2$ we get
\begin{equation}
\begin{split}
||(w_1,w_2)||_{g_m}^2=&\frac{1+4mV(1+mU+mV)}{1+2m(U+V)}e^{2m(V-U)} |w_1|^2+\\
&+\frac{4m(1+mU+mV)}{1+2m(U+V)}(z_2\bar z_1w_1\bar w_2+z_1\bar z_2w_2\bar w_1)+\\
&+\frac{1+4mU(1+mU+mV)}{1+2m(U+V)}e^{2m(U-V)}|w_2|^2=\\
=&\frac{e^{2m(U-V)}}{1+2m(U+V)} |w_1|^2+\frac{e^{2m(U-V)}}{1+2m(U+V)}|w_2|^2\\
&+\frac{4m(1+mU+mV)}{1+2m(U+V)}|z_2w_1+z_1w_2|^2,
\end{split}\nonumber
\end{equation}
where we used the identity $z_2\bar z_1 w_1\bar w_2+z_1\bar z_2 w_2\bar w_1=|z_2w_1+z_1w_2|^2-|z_2|^2|w_1|^2-|z_1|^2|w_2|^2$. Our problem reduces then to show that $(\C^2,\tilde g_m)$ is complete (see e.g. \cite[Ex. 7 p. 153]{docarmo}). Actually, it is enough to show that $(S, \tilde g_m|_S)$ is complete, where $S$ is the surface of real dimension $2$ defined by $$S=\left\{(z_1,z_2)\in\C^2\,|\ \mathrm{Im}(z_1)=0,\mathrm{Im}(z_2)=0\right\}.$$ In fact, since $\tilde g_m$ depends only on the modules of the variables, fixed $\theta=(\theta_1,\theta_2)\in\R^2$, the maps $\varphi_\theta$, $\psi\!:\C^2\f\C^2$ defined by $\varphi_\theta(z_1,z_2)=(e^{i\theta_1}z_1,e^{i\theta_2}z_2)$ and $\psi(z_1,z_2)= (\bar z_1, \bar z_2)$, are isometries. Thus, $S$ is totally geodesic in $(\C^2,\tilde g_m)$ being the fixed points' locus of $\psi$, and every geodesic through the origin of $(\C^2,\tilde g_m)$ can be viewed as a geodesic of $S$, since any vector tangent at the origin of $\C^2$ is isometric through a map $\varphi_\theta$, for some $\theta$, to a vector of $T_{(0,0)}S$.

Let $\mu_1=\mathrm{Re}(z_1)$, $\mu_2=\mathrm{Re}(z_2)$. For $(z_1,z_2)\in S$ we get $z_1=\mu_1$, $z_2=\mu_2$ and by $|\mu_1|=e^{m(u^2-v^2)}u$, $|\mu_2|=e^{m(v^2-u^2)}v$, it follows that
$$d\mu_1=\pm\left[(2mu^2+1)du-2muvdv\right]e^{m(u^2-v^2)},$$
$$d\mu_2=\pm\left[(2mu^2+1)dv-2muvdu\right]e^{m(v^2-u^2)},$$
which implies
\begin{equation}
\begin{split}\nonumber
\tilde g_m|_S=&\frac{\left[(2mu^2+1)du-2muvdv\right]^2+\left[(2mv^2+1)dv-2muvdu\right]^2}{1+2m(u^2+v^2)}\\
=&\frac{(1+2mu^2)^2+4m^2u^2v^2}{1+2m(u^2+v^2)}du^2+\frac{(1+2mv^2)^2+4m^2u^2v^2}{1+2m(u^2+v^2)}dv^2+\\
&-\frac{8muv(1+mu^2+mv^2)}{1+2m(u^2+v^2)}dudv.
\end{split}
\end{equation}
We claim that
\begin{equation}\label{lastcomplete}
\tilde g_m|_S\geq \frac{du^2+dv^2}{1+2m(u^2+v^2)}.
\end{equation}
Observe that if (\ref{lastcomplete}) holds we are done, since in polar coordinates the right-hand side reads $(d\rho^2+\rho^2d\theta^2)/(1+2m\rho^2)$, and
$$\int_0^{+\infty}||\alpha'(t)||_{\tilde g_m|_S}dt\geq\int_{\rho_0}^{+\infty}\frac{d\rho}{\sqrt{1+2m\rho^2}}=+\infty,$$
for any divergent curve $\alpha\!:[0,+\infty)\f S$ (see e.g. \cite[Ex. 5 p. 153]{docarmo}). In order to verify (\ref{lastcomplete}), observe that, for any vector $(\alpha,\beta)$ tangent to $S\equiv \R^2$ at $(u,v)$, it is equivalent to
\begin{equation}
\begin{split}\label{diseq}
&\left((1+2mu^2)^2+4m^2u^2v^2-1\right)\alpha^2+\left((1+2mv^2)^2+4m^2u^2v^2-1\right)\beta^2+\\
&-8muv(1+mu^2+mv^2)\alpha\beta\geq 0.
\end{split}
\end{equation}
If $\beta=0$, it holds true. Assume $\beta\neq 0$. Setting $\gamma=\alpha/\beta$ we get a second order equation in $\gamma$ whose discriminant is $0$. The conclusion follows by noticing that the coefficient of the leading term, namely $(1+2mu^2)^2+4m^2u^2v^2-1$, is nonegative and it vanishes iff $u=0$, case for which (\ref{diseq}) reads $(1+2mv^2)^2-1\geq 0$.
\end{proof}

\begin{Alem}\label{curvature}
Fix $m\geq 0$ and let $g=g_m$ be the corresponding  LeBrun  metric  on $\C^2$. Then the squared norm of its curvature tensor and its Laplacian are given respectively by:
\begin{equation}\label{|R|^2}
|R|^2=\frac{96m^2}{(1+2m(U+V))^8},
\end{equation}
and
\begin{equation}\label{L|R|^2}
\Delta|R|^2=\frac{3072\,m^3\,(7m(U+V)-1)}{(1+2m(U+V))^{11}}.
\end{equation}
\end{Alem}
\begin{proof}
Recall that the curvature tensor for a given K\"ahler metric $g=(g_{i\bar j})$ on a $n$-dimensional complex manifold is given by
\begin{equation}\label{curvature2}
R_{i\bar j k\bar l}=-\frac{\de^2 g_{i\bar l}}{\de z_k\de\bar z_j}+\sum_{p q=1}^n g^{p\bar q}\frac{\de g_{i\bar p}}{\de z_k}\frac{\de g_{q\bar l}}{\de\bar z_j},
\end{equation}
and, in accordance with the general definition of norm of a complex tensor,
 (see e.g. \cite[p. 127]{zhe}) we have:
\begin{equation}\label{normcurv}
|R|^2=\sum_{i,j,k,l,p,q,r,s=1}^n {g^{p\bar i}}{g^{j\bar q}}{g^{r\bar k}}{g^{l\bar s}}R_{i\bar j k\bar l}\overline{R_{p\bar q r\bar s}}.
\end{equation}
Moreover the Laplacian $\Delta|R|^2$ reads
\begin{equation}\label{laplaciannormcurv}
\Delta|R|^2=\sum_{i,j=1}^n {g^{i\bar j}}\frac{\de^2 |R|^2}{\de z_i\de\bar z_j}.
\end{equation}
Here $g^{i\bar j}$, $i, j=1, \dots ,n$ are the entries  of the inverse of
$(g_{i\bar j})$, namely $\sum_{j =1}^{n}g^{i\bar
j}g_{j\bar k}=\delta_{ik}$.

Observe that by (\ref{metricg}), since $\det(g)=1$, we have
\begin{equation}\label{metricginv}
\left(
\begin{array}{cc}
g^{1\bar 1} & g^{1\bar 2}   \\
g^{2\bar 1}  & g^{2\bar 2}
\end{array}
\right)=\left(
\begin{array}{cc}
\Phi_{22}\,x_2+\Phi_2 &-\Phi_{12}\,\bar z_1 z_2    \\
- \Phi_{12}\, z_1\bar z_2 &  \Phi_{11}\,x_1+\Phi_{1}
\end{array}
\right).
\end{equation}
Further, by (\ref{metricg}) we obtain
\begin{equation}\label{firstderiv}
\begin{split}
\frac{\de g_{k\bar k}}{\de z_j}=&\Phi_{kkj}\bar z_j x_k+\Phi_{kk}\bar z_j\delta_{jk}+\Phi_{kj}\bar z_j,\\
\frac{\de g_{k\bar j}}{\de z_j}=&\Phi_{kj}\bar z_k+\Phi_{kjj}\bar z_k x_j,\quad (\textrm{for}\ k\neq j),\\
\frac{\de g_{j\bar k}}{\de z_j}=&\Phi_{jkj}\bar z_j^2z_k,\quad (\textrm{for}\ k\neq j).
\end{split}
\end{equation}
where $\Phi_{jkl}=\frac{\partial^3\Phi}{\partial z_j\partial z_k\partial z_l}, j, k, l=1, 2.$
It  follows easily by the previous computations that
$$\Phi_{111}=\frac{4m^2e^{6m(V-U)}}{(1+2m(U+V))^3}(3+4m(U+2V)),$$
$$\Phi_{222}=\frac{4m^2e^{6m(U-V)}}{(1+2m(U+V))^3}(3+4m(2U+V)),$$
$$\Phi_{112}=-\frac{4m^2e^{2m(V-U)}}{(1+2m(U+V))^3}(1+4mV),$$
$$\Phi_{221}=-\frac{4m^2e^{2m(U-V)}}{(1+2m(U+V))^3}(1+4mU).$$
Finally, by (\ref{firstderiv}) we get, for  $ j, k=1, 2$,
\begin{equation*}
\begin{split}
\frac{\de^2 g_{k\bar k}}{\de z_k\de \bar z_k}=&\Phi_{kkkk}x^2_k+4\Phi_{kkk}x_k+2\Phi_{kk},\\
\frac{\de^2 g_{k\bar j}}{\de z_k\de \bar z_j}=\frac{\de^2 g_{k\bar k}}{\de z_j\de \bar z_j}=&\Phi_{kkjj}x_kx_j+\Phi_{kkj}x_k+\Phi_{kjj}x_j+\Phi_{kj},\quad (\textrm{for}\ k\neq j)\\
\frac{\de^2 g_{j\bar k}}{\de z_k\de \bar z_k}=\frac{\de^2 g_{k\bar k}}{\de z_k\de \bar z_j}=&\Phi_{kkkj}z_k^2\bar z_k\bar z_j+2\Phi_{kkj}z_k\bar z_j,\quad (\textrm{for}\ k\neq j),\\
\frac{\de^2 g_{j\bar k}}{\de z_j\de \bar z_k}=&\Phi_{kkjj}z_k^2\bar z_j^2,\quad (\textrm{for}\ k\neq j),
\end{split}
\end{equation*}
where

$\frac{\partial^4\Phi}{\partial x_1^4}=\Phi_{1111}=-\frac{16m^3e^{8m(V-U)} \left(8+43mV+19mU+12m^2\left(4UV+5V^2+U^2\right)\right)}{(1+2m(U+V))^5},$
\vskip 0.1cm
$\frac{\partial^4\Phi}{\partial x_2^4}=\Phi_{2222}=-\frac{16m^3e^{8m(U-V)}\left(8+43mU+19mV+12m^2\left(4UV+5U^2+V^2\right)\right)}{(1+2m(U+V))^5},$
\vskip 0.1cm
$\frac{\partial^4\Phi}{\partial x_1^3\partial x_2}=\Phi_{1112}=\frac{16m^3e^{4m(V-U)}\left(2+mU+13mV+24m^2V^2\right)}{(1+2m(U+V))^5},$
\vskip 0.1cm
$\frac{\partial^4\Phi}{\partial x_1\partial x_2^3}=\Phi_{2221}=\frac{16m^3e^{4m(U-V)}\left(2+mV+13mU+24m^2U^2\right)}{(1+2m(U+V))^5},$
\vskip 0.1cm
$\frac{\partial^4\Phi}{\partial x_1^2\partial x_2^2}=\Phi_{1122}=\Phi_{2211}=-\frac{8m^3 \left(1+2m(U+V)+8m^2(U-V)^2\right)}{(1+2m(U+V))^4}.$

\vskip 0.1cm
Thus, it is not hard to see using  (\ref{curvature}) that
\begin{equation}
{\small \begin{split}
R_{1\bar 1 1\bar 1}=\frac{4 me^{4m(V-U)}}{(1+2m(U+V))^5} (&16m^4V(V^3-4UV^2+U^2V)+\\
&+32m^3V^2(V-2U)+8m^2V(3V-2U)+8mV+1),
\end{split}\nonumber}
\end{equation}
\begin{equation}
{\small\begin{split}
R_{1\bar 1 1\bar 2}=R_{1\bar 2 1\bar 1}=\overline{R_{2\bar 1 1\bar 1}}=\overline{R_{1\bar 1 2\bar 1}}=&\frac{8m^2z_1\bar z_2e^{2m(V-U)}}{(1+2m(U+V))^5}(\, 8m^3V(V^2-4UV+U^2)+\\
&\qquad\quad+4m^2V(V-5U)-2m(V+2U)-1)),
\end{split} \nonumber}
\end{equation}
\begin{equation}
{\small \begin{split}
R_{1\bar 1 2\bar 2}=R_{2\bar 2 1\bar 1}=&{R_{1\bar 2 2\bar 1}}={R_{2\bar 1 1\bar 2}}=\frac{4 m}{(1+2m(u+v))^5}(16m^4uv(u^2+v^2-4uv)+\\
&\quad \quad -16m^3uv (u+v)-4m^2(u^2+v^2+4uv)-4m(u+v)-1),
\end{split} \nonumber}
\end{equation}
\begin{equation}
\small\begin{split}
R_{1\bar 2 1\bar 2}&=\overline{R_{2\bar 1 2\bar 1}}=\frac{32(2m^2(V^2+U^2-4UV)-2m(U+V)-1))\bar z_1^2z_2m^3}{(1+2m(U+V))^5},
\end{split} \nonumber
\end{equation}
\begin{equation}
{\small\begin{split}
R_{2\bar 2 2\bar 1}=R_{2\bar 1 2\bar 2}=\overline{R_{1\bar 2 2\bar 2}}=\overline{R_{2\bar 2 1\bar 2}}=&\frac{8 m^2z_2\bar z_1e^{2m(U-V)}}{(1+2m(U+V))^5}(8m^3U(U^2-4VU+V^2)+\\
&\quad \quad+4m^2U(U-5V)-2m(U+2V)-1),
\end{split} }\nonumber
\end{equation}
 \begin{equation}
{\small\begin{split}
R_{2\bar 2 2\bar 2}=\frac{4me^{4m(U-V)}}{(1+2m(U+V))^5}(&16m^4U(U^3-4VU^2+V^2U)+\\
&+32m^3U^2(U-2V)+8m^2U(3U-2V)+8mU+1).
\end{split} \nonumber}
\end{equation}
Formulas  (\ref{|R|^2})  and (\ref{L|R|^2})  can be obtained, after a long but straightforward computation,  by substituting
the previous expressions into (\ref{normcurv}) and (\ref{laplaciannormcurv}).
\end{proof}

\end{document}